\documentclass[11pt,bezier]{article}
\usepackage{tikz}
\usepackage{float}
\usepackage{amsmath,amssymb,amsfonts,euscript,graphicx}

\newtheorem{preproof}{{\bf \indent Proof.}}

\newenvironment{proof}[1]{\begin{preproof}{\rm
               #1}\hfill{$\Box$}}{\end{preproof}}

\newtheorem{example}{\bf\indent Example}[section]
\newtheorem{thm}{{\bf\indent Theorem}}[section]
\newtheorem{prop}{\bf\indent Proposition}[section]
\newtheorem{remark}{\bf\indent Remark}[section]
\newtheorem{lem}{\bf\indent Lemma}[section]

\title{\bf \large Exploring Metric and Strong Metric Dimensions in\\ Inclusion Ideal Graphs of Commutative Rings}

\author{{\normalsize  {\sc E. Dodongeh${}^{\mathsf{a}}$,  A. Moussavi${}^{\mathsf{a}}$,    R.  Nikandish${}^{\mathsf{b}}$}
}\vspace{3mm}\\
{\footnotesize{${}^{\mathsf{a}}$\it Department of Mathematics, University of Tarbiat Modarres,
Tehran, Iran}}\\
{\footnotesize{${}^{\mathsf{b}}$\it  Department of Mathematics,
Jundi-Shapur University of Technology,  Dezful,
Iran}}\\
{\footnotesize{${}^{\mathsf{}}$\it}}\\
{\footnotesize{$\mathsf{e.dodongeh@modares.ac.ir}$\quad\quad $\mathsf{moussavi.a@modares.ac.ir}$ \quad\quad
$\mathsf{r.nikandish@ipm.ir}$\quad\quad}}}

\begin{document}

\maketitle

\begin{abstract}
{\small This paper investigates the inclusion ideal graph of a commutative unitary ring $R$, denoted as $In(R)$, which is defined by its vertices representing all non-trivial ideals of $R$. The edges of the graph connect two distinct vertices if one ideal is a proper subset of the other. We explore the metric dimension of $In(R)$, providing a comprehensive characterization of its resolving graph. Additionally, we compute the strong metric dimension of $In(R)$, illustrating the implications of our findings on the structure of the graph and its applications in the study of ideal theory within commutative rings.}
\end{abstract}
\begin{center}\section{Introduction}\end{center}
Metric and strong metric dimension in a graph are examples of NP-hard problems in discrete structures that have found several applications in computer science, mechanical engineering, optimization, chemistry etc. Although many important works have been done by graph theorists in computing metric and strong metric dimension, they are still two of the most active research areas in graph theory. For the most recent study in this field see \cite{abr, bai, ghala, Ji, Ve}. In addition to a wide range of applications, the complexity of computations  has sparked considerable interest in characterizing these invariants for graphs  associated with algebraic structures. Some
examples in this direction may be found in \cite{ali, dol, dol2,  dol3, ma, nili, nili2, Pirzada1, Pirzada2, zai}.  This paper has such a theme and aims to compute the metric and strong metric dimension in inclusion ideal graphs of commutative rings.
\par
 For graph theory terminology, we follow \cite{west}. Let $G=(V,E)$ be a graph with $V=V(G)$ as the vertex set and $E=E(G)$ as the edge set. A complete graph of order $n$  is denoted by $K_n$. Also, distance between two distinct vertices $x$ and $y$ is denoted by $d(x,y)$. By diam$(G)$, we mean the diameter of $G$. If a graph $H$ is a subgraph of $G$, then we write $H \subseteq G$. Moreover, the induced subgraph by $V_0\subseteq V$ is  denoted by $G[V_0]$. The open and closed neighborhood of the vertex $x$ are denoted by  $N(x)$ and $N[x]$, respectively. The independence number and vertex cover number of the graph $G$ are denoted by $\beta(G)$ and  $\alpha(G)$, respectively.
 Let $S=\{v_1,v_2,\dots,v_k\}$ be an ordered subset of $V$ and $v\in V\setminus S$. Then the representation vector of  $v$ with respect to $S$ is denoted by $D(v|S)$ which is defined as follows: $D(v|S)=(d(v,v_1),d(v,v_2),\dots, d(v,v_k))$.  An ordered subset $S\subseteq V(G)$ is called \textit{resolving} provided that distinct vertices out of $S$ have different representation vectors with respect to $S$.  Any
resolving set of minimum cardinality is called \textit{metric basis for} $G$, and its cardinal number is called the \textit{metric dimension of} $G$. We denote the metric dimension of $G$  by $dim_M(G)$.
 Two different vertices $u,v$ \textit{are mutually maximally distant} if $d(v, w) \leq d(u, v)$, for every $w \in N(u)$ and $ d(u, w) \leq d(u, v)$, for every $w \in N(v)$. For a graph $G$, \textit{the strong resolving graph of} $G$, is denoted by $G_{SR}$ and its vertex and edge set are defined as follow:  
$V(G_{SR})= \lbrace u \in V (G)|\,there~exists ~v \in V (G)  ~such~that ~u,  v ~are ~mutually ~maximally ~distant \rbrace$ and  $uv \in E(G_{SR})$  if  and  only  if $u$  and $v$ are  mutually maximally distant. Two vertices $u$ and $v$ are \textit{strongly resolved} by some vertex $w$ if either 
$d(w, u)$ is  equal to $d(w, v) + d(v, u)$  or $d(w, v)$ is equal to $d(w, u) + d(v, u)$. A set $W$ of vertices
is a \textit{strong resolving set of} $G$ if every two distinct vertices of $G$ are strongly
resolved by some vertex of $W$ and a minimum strong resolving set is called
 \textit{strong metric basis} and its cardinality is \textit{the strong metric dimension of} 
$G$. We denoted the strong metric dimension of $G$, by $sdim(G)$.

\par
Throughout this paper, all rings are assumed to be commutative with identity. The set of all non-trivial ideals of $R$ is denoted by  $I(R)$. The ring $R$ is called  \textit{reduced} if it has no nilpotent elements other than $0_R$. For undefined notions in ring theory, we refer the reader to \cite{ati}.
\par
\textit{The inclusion ideal graph of a  ring} $R$, denoted by $In(R)$, is a graph whose vertex set is $I(R)$ and two distinct vertices are adjacent if and only if one of them is properly contained in the other. This graph was first introduced and studied by Akbari et.al in  \cite{Akbari} and several interesting properties were obtained. Subsequently, this concept has been the subject of many researches; see for instance \cite{das, ou, wang}.
In this paper, we characterize the metric dimension of  $In(R)$. Furthermore, we investigate the structure of the strong resolving graph of $In(R)$ and compute  $sdim(In(R))$ as an application.
\section{ $dim_{M}(In(R))$ and $sdim(In(R))$, when $R$ is reduced}
 \noindent
 In this section, we first show that $dim_M(In(R))$ is finite if and only if $|I(R)|< \infty$.  Then we provide some metric and strong metric dimension formulas for $dim_M(In(R))$ and $sdim(In(R))$, when $R$ is a reduced ring.

We fix the following notations.
\begin{remark}
\label{dimfin}
{\rm Let $R \cong \prod_{i=1}^{n}R_{i}$, where $R_{i}$ is a ring for every $1\leq i \leq n$, and $I=I_{1}\times \cdots \times I_{n} \in V(In(R))$. We adopt the following notations: \\
$1)$ By $I^{c}=I_{1}^{c}\times \cdots \times I_{n}^{c}$, we mean a vertex of $In(R)$ such that $I_{i}^{c}=R_{i}$ if and only if $I_{i}=0$ for every $1\leq i \leq n$.\\
$2)$ If $R_i$ is a field, then $X_{i}=0\times \cdots \times 0\times R_{i} \times 0 \times\cdots \times 0$, where the field $R_{i}$ is in the $i$-th position (We call $X_i$ the $i$-th minimal ideal, if every $R_i$ is a field).\\
$3)$ By $M$, we mean the following subset of $V(In(R))$: $$M=\lbrace I=I_{1}\times \cdots \times I_{m}|\, I_{i}=0 ~or~ I_{i}=R_{i}~ for~ every~ 1\leq i \leq n\rbrace.$$}
\end{remark}

\begin{prop}\label{dimfinite}
Let $R$ be a ring that is not a field. Then $dim_{M}(In(R))< \infty$ if and only if $|I(R)|< \infty$.
\end{prop}
\begin{proof}
{First assume that $dim_M(In(R))$ is finite and $W=\{W_1,\ldots,W_n\}$ is a metric basis for $In(R)$, where $n$ is a non-negative integer. By \cite[Theorem\,  2.1]{Akbari}, there are only $3^n$ possibilities for $D(X|W)$, for every $X \in V(In(R))\setminus W$. Thus $|V(In(R)| \leq 3^{n}+n$ and hence $R$ has finitely many ideals. The converse implication is clear.}
\end{proof}

By Proposition \ref{dimfinite}, to express metric and strong metric dimensions in some explicit formulas, it is enough to consider rings with finitely many ideals. Therefore, from now on, we suppose that all rings $R$ have finitely many ideals. It is well-known that reduced Artinian rings are the direct product of finitely many fields.

Let $n\geq 3$ be a positive integer. In the next theorem, $dim_M(In(\prod_{i=1}^{n}\mathbb{F}_{i}))$ is determined. For this aim, the following lemma is needed. 

\begin{lem}\label{lemma2d}
Let  $n\geq 3$ be a positive integer and $R\cong \prod_{i=1}^{n}\mathbb{F}_{i}$, where $\mathbb{F}_{i}$ is a field for every $1\leq i \leq n$. Then $diam(In(R))=3$.
\end{lem}
\begin{proof}
{By \cite[Theorem\,  2.1]{Akbari}, $In(R)$ is a connected graph and $diam(In(R))\leq 3$. Now, let  $J= X_1^{c}$.  Since $X_1 \nsim J$, $d(X_1 ,J)\geq 2$. If $d(X_1 ,J)=2$, then there exists a vertex $T$ such that $X_1 \sim T \sim J$ is the shortest path between $X_1 $ and $J$. Since  $X_1 \sim T$, $X_1 \subset T$ or $T\subset X_1 $. If $T\subset X_1 $, then clearly $T \nsim J$. Thus $X_1 \subset T$. A similar argument shows that $J\subset T$, which means $T=R$, a contradiction. Thus $X_1 \sim \mathbb{F}_{1}\times \mathbb{F}_{2}\times 0\times \cdots \times 0 \sim X_{2}\sim J$ is the shortest path between $X_1 $ and $J$, and so $diam(In(R))=3$.}
\end{proof}


\begin{thm}\label{dimprod}
Suppose that $R\cong \prod_{i=1}^{n}\mathbb{F}_{i}$,  where $\mathbb{F}_{i}$ is a field for every $1\leq i \leq n$ and $n\geq 3$ be a positive integer.  Then the following statements hold:

$1)$  If $n \leq 4$, then $dim_M(In(R))=n-1$.\\

$2)$  If $n \geq 5$, then $dim_M(In(R))=n$.
\end{thm}
\begin{proof}
{$(1)$ 
If $n=3$, then $In(R)=C_6$ and thus $dim_M(In(R))=2$.\\
\unitlength=1.5mm
\begin{center}
\begin{figure}[H]
 	
\begin{minipage}{1.2\textwidth}
\begin{center}
\begin{tikzpicture}
  [scale=0.25,every node/.style={circle,fill=black,inner sep=0pt},very thick]
  \node [fill=red, label=above:{\tiny $X_{1}$},text width=2mm] (n1) at (0,5) {};
  \node [label=above:{\tiny $X_{2}$},text width=2mm] (n2) at (16,5)  {};
  \node [label=below:{\tiny $X_{3}$},text width=2mm] (n3) at (8,0)  {};
  \node [label=below:{\tiny $X_{1}^{c}$},text width=2mm]  (n4) at (16,0)  {};
  \node [fill=red, label=below:{\tiny $X_{2}^{c}$},text width=2mm] (n5) at (0,0)  {};
  \node [label=above:{\tiny $X_{3}^{c}$},text width=2mm] (n7) at (8,5)  {};
  \foreach \from/\to in {n2/n4,n4/n3,n3/n5,n5/n1,n1/n7,n7/n2}
    \draw (\from) -- (\to);
\end{tikzpicture}
\end{center}

 \caption{$ In(R) $} \label{figure:fr}
\end{minipage}
\end{figure}
\end{center}

If $n=4$, then since $diam(In(R))=3$ and $|V(In(R))|=14$, we have  $dim_M(In(R))\geq 3$.
Now we show that  $dim_M(In(R))\leq 3$.

If we put $W=\lbrace X_{1},X_{2},X_{3} \rbrace$, then we get\\
$$ D(\mathbb{F}_{1}\times \mathbb{F}_{2}\times \mathbb{F}_{3}\times 0)|W)=(1,1,1) \,\,\,\,\  D(\mathbb{F}_{1}\times \mathbb{F}_{2}\times 0 \times \mathbb{F}_{4})|W)=(1,1,3) \,\,\,\,\   D(\mathbb{F}_{1}\times  0 \times \mathbb{F}_{3} \times \mathbb{F}_{4})|W)=(1,3,1)  $$
$$ D(0 \times \mathbb{F}_{2}\times \mathbb{F}_{3} \times \mathbb{F}_{4})|W)=(3,1,1) \,\,\,\,\  D(\mathbb{F}_{1}\times \mathbb{F}_{2}\times 0 \times0)|W)=(1,1,2) \,\,\,\,\   D(\mathbb{F}_{1}\times 0 \times \mathbb{F}_{3}\times 0 )|W)=(1,2,1)  $$
$$ D(0\times \mathbb{F}_{2}\times \mathbb{F}_{3}\times 0 )|W)=(2,1,1) \,\,\,\,\  D(\mathbb{F}_{1}\times 0 \times 0 \times \mathbb{F}_{4})|W)=(1,2,2) \,\,\,\,\   D(0\times \mathbb{F}_{2}\times 0 \times \mathbb{F}_{4})|W)=(2,1,2)  $$
$$ D(0 \times 0 \times \mathbb{F}_{3} \times \mathbb{F}_{4})|W)=(2,2,1)  \,\,\,\,\  D(0 \times 0 \times 0 \times \mathbb{F}_{4})|W)=(2,2,2).$$
This shows that $W$ is a resolving set for $In(R)$ and hence $dim_M(In(R))\leq 3$.

$(2)$ We show that $dim_M(In(R))=n$, for every $n\geq 5$. Indeed, we have the following claims:

\textbf{Claim 1.}
$dim_M(In(R))\geq n$.\\
Proposition \ref{dimfinite} shows that  $dim_M(In(R))$ is finite. Let
  $W=\lbrace W_1,\ldots,W_k\rbrace$ be a metric basis for $In(R)$, where $k$ is a positive integer. By  Lemma \ref{lemma2d}, $d(I,I^c)=3$, for every $I\in V(In(R))$ (Indeed,  $d(I, J)=3$ iff $J=I^c$). Hence there is at most one 3 in $D(I|W)$ and for the other components there are $2$ possibilities,  for each $I\in V(In(R))$. Thus if $I^c\in W$, then there are $2^{k-1}$ possibilities for $D(I|W)$ and otherwise there are $2^k$ possibilities, for every $X\in V(In(R))\setminus W$. Since $|V(In(R))|=2^{n}-2$ and $|V(In(R))|-2k \leq 2^k$, we have $ 2^{n}\leq 2^{k}+2k+2$. Since $n\geq 5$,  we conclude that $k\geq n$. Therefore $dim_M(In(R))\geq n$.

\textbf{Claim 2.}
$dim_M(In(R))\leq n$.\\
Let  $W=\{X_1,\ldots,X_n\}$, where $X_i$ is the i-th minimal ideal, for every $1\leq i\leq n$ (see Remark \ref{dimfin}).
We show that $W$ is a resolving set for $In(R)$. Let $I,J\in V(In(R))\setminus W$  and $I\neq J$. We show that $D(I|W) \neq D(J|W)$. Since $diam(In(R))=3$, we consider the following cases:

$(a)$  $I=X_i^c$, for some $1\leq i \leq n$. In this case, each component of $D(I|W)$ is $1$ if and only if the corresponding component in $I$ is a field. Moreover, the $i$-th component of $D(I|W)$ is $3$.

$(b)$  $I\neq X_i^c$, for every $1\leq i \leq n$.
In this case each component of $D(I|W)$ is $1$ if and only if the corresponding component in $I$ is a field, and each component of $D(I|W)$ is $2$ if and only if the corresponding component in $I$ is zero.

By cases $(a)$ and  $(b)$, if $I\neq J$, then $D(I|W)\neq D(J|W)$. Thus $W$ is a resolving set for $In(R)$. Therefore $dim_M(In(R))\leq n$.

By Claims $1,2$, $dim_M(In(R))=n$, for $n\geq 5$.
}
\end{proof}

The next goal of this section is to find $sdim(In(R))$, where $R$ is a direct product of finitely many fields. To this end, we need a series of lemmas.

\begin{lem}\label{Oellermann} {\rm (\cite[Theorem 2.1]{oller})} For any connected graph $G$, $sdim_M(G)=\alpha(G_{SR})$.
\end{lem}

\begin{lem}\label{Gallai}
{\rm (Gallai$^{^,}$s theorem)} For any graph $G$ of order $n$, $\alpha(G)+\beta(G)=n$.
\end{lem}

\begin{lem}\label{lemma2g}
Let  $n\geq 3$ be a positive integer and $R\cong \prod_{i=1}^{n}\mathbb{F}_{i}$, where $\mathbb{F}_{i}$ is a field for every $1\leq i \leq n$. Then the following statements hold.\\
$1)$ $V(In(R))=V(In(R)_{SR})$.\\
$2)$ Suppose that $I, J \in V(In(R)_{SR})$, then  $IJ \in E(In(R)_{SR})$ if and only if $I=J^{c}$ or $IJ, IJ^{c} \notin E(In(R)).$
\end{lem}
\begin{proof}
{
$1)$ For every $I=I_{1}\times\cdots \times I_{n}\in V(In(R))$, Lemma \ref{lemma2d} implies that $d(I,I^{c})=3=diam(In(R))$. Thus $I, I^{c}$ are mutually maximally distant and so $I\in V(In(R)_{SR})$ i.e., $V(In(R))=V(In(R)_{SR}).$\\
$2)$ First suppose that  $I=J^{c}$ or $IJ, IJ^{c} \notin E(In(R))$. If $I=J^{c}$, then obviously $IJ\in E(In(R)_{SR})$. Hence one may suppose that $I\neq J^{c}$ and $IJ, IJ^{c} \notin E(In(R))$. Since $IJ\notin E(In(R))$ and $I\neq J^{c}$,  $d_{In(R)}(I,J)=2$. Also  $I\nsim J^{c}$ implies that $d(V,J)\leq d(I,J)$, for every $V\in N(I)$. Moreover, $d(U,I)\leq d(I,J)$, for every $U\in N(J)$. Therefore, $I, I^{c}$ are mutually maximally distant, thus $IJ \in E(In(R)_{SR})$.\\
Conversely, suppose that $IJ \in  E(In(R)_{SR}$, for some $I, J \in V(In(R)_{SR})$ and $I\neq J^{c}$. Then clearly $I\nsim J$ and if $I\sim J^{c}$, then $d_{In(R)}(J,J^{c})=3 > d(I,J)$, and so $I, I^{c}$ are not mutually maximally distant, a contradiction. This completes the proof.
}
\end{proof}

\begin{lem}\label{dimprod2}
Let  $n\geq 3$ be a positive integer and $R\cong \prod_{i=1}^{n}\mathbb{F}_{i}$, where $\mathbb{F}_{i}$ is a field for every $1\leq i \leq n$. Then
$In(R)_{SR}= H + \underbrace{K_{2}+ \cdots +K_{2}}_{n\, times}$, where $H$ is a connected graph.
\end{lem}
\begin{proof}
{By Lemma \ref{lemma2g}, $V(In(R))=V(In(R)_{SR})$. Assume that $I=I_{1}\times\cdots \times I_{n}\in V(In(R)_{SR})$. Let $NZC(I)$ be the number of zero components in $I$. Obviously, $1 \leq NZC(I) \leq n-1$. Assume that \\
$A_{1}=\lbrace I \in V(In(R)_{SR})| NZC(I)=1\rbrace$,\\
$A_{2}=\lbrace I \in V(In(R)_{SR})| NZC(I)=2\rbrace$,\\
\vdots
\\
and $A_{n-1}=\lbrace I \in V(In(R)_{SR})| NZC(I)=n-1\rbrace$.\\
Let $I \in A_{1}$. Then for every $J\neq I^{c}$ with $IJ \notin E(In(R))$ we have $I^{c} \subset J$ and so by Lemma \ref{lemma2g}, $IJ \notin E(In(R)_{SR}$. Similarly, for every $I \in A_{n-1}$ the only vertex maximally distant from $I$ is  $I^{c}$ and vice versa, so  $I$ is only mutually maximally distant from $I^{c}$. Since $|A_{1}|=|A_{n-1}|=n$, $In(R)_{SR}$ contains $n$ copies of $K_{2}$.\\
Now, we show that $H= In(R)_{SR}[S]$ is a connected graph. For every two distinct vertices $I, J \in S=V(In(R)_{SR})\setminus (A_{1} \cup A_{n-1})$, if $I \nsim J$ and $I\nsim J^{c}$  in $In(R)$, then $d_{H}(I,J)= 1$ , else $d_{H}(I,J)= 2$, so $H= In(R)_{SR}[S]$ is a connected graph.}
\end{proof}

The next example explains Lemma \ref{dimprod2} in case $n=4$.

\begin{example}
{{\rm
Suppose that $R\cong \prod_{i=1}^{4}\mathbb{F}_{i}$,  where $\mathbb{F}_{i}$ is a field for every $1\leq i \leq 4$. Thus $|V(In(R))|=14$. Let}\\
$ V_{1}=\mathbb{F}_{1}\times \mathbb{F}_{2}\times \mathbb{F}_{3}\times 0, \,\,\,\,\  V_{2}=\mathbb{F}_{1}\times \mathbb{F}_{2}\times 0 \times \mathbb{F}_{4}, \,\,\,\,\   V_{3}=\mathbb{F}_{1}\times  0 \times \mathbb{F}_{3} \times \mathbb{F}_{4},  $$\\ 
$$   V_{4}=0 \times \mathbb{F}_{2}\times \mathbb{F}_{3} \times \mathbb{F}_{4},\,\,\,\,\ V_{5}=\mathbb{F}_{1}\times \mathbb{F}_{2}\times 0 \times0, \,\,\,\,\ V_{6}=\mathbb{F}_{1}\times 0 \times \mathbb{F}_{3}\times 0,$$\\
$$ V_{7}=0\times \mathbb{F}_{2}\times \mathbb{F}_{3}\times 0, \,\,\,\,\    V_{8}=\mathbb{F}_{1}\times 0 \times 0 \times \mathbb{F}_{4}, \,\,\,\,\   V_{9}=0\times \mathbb{F}_{2}\times 0 \times \mathbb{F}_{4}, $$\\
$$ V_{10}=0 \times 0 \times \mathbb{F}_{3} \times \mathbb{F}_{4},  \,\,\,\,\  V_{11}=0 \times 0 \times 0 \times \mathbb{F}_{4},$$  
$$\,\,\ V_{12}=0 \times 0 \times \mathbb{F}_{3}\times 0 ,$
$  V_{13}=0 \times \mathbb{F}_{2}\times 0 \times 0,  \,\,\,\,\  V_{14}=\mathbb{F}_{1} \times 0 \times 0 \times 0$.
\unitlength=1.5mm
\begin{center}
\begin{figure}[H]
 	
\begin{minipage}{.2\textwidth}
\begin{tikzpicture}
  [scale=0.25,every node/.style={circle,fill=black,inner sep=0pt},very thick]
  \node [label=above:{\tiny $V_{1}$},text width=2mm] (n1) at (0,5) {};
  \node [label=above:{\tiny $V_{2}$},text width=2mm] (n2) at (2,5)  {};
  \node [label=above:{\tiny $V_{3}$},text width=2mm] (n3) at (4,5)  {};
  \node [label=above:{\tiny $V_{4}$},text width=2mm]  (n4) at (6,5)  {};
  \node [label=left:{\tiny $V_{5}$},text width=2mm] (n5) at (-12,0)  {};
  \node [label=left:{\tiny $V_{6}$},text width=2mm] (n6) at (-8,0)  {};
  \node [label=left:{\tiny $V_{7}$},text width=2mm] (n7) at (-4,0)  {};
  \node [label=right:{\tiny $V_{8}$},text width=2mm] (n8) at (8,0)  {};
  \node [label=right:{\tiny $V_{9}$},text width=2mm] (n9) at (12,0)  {};
  \node [label=right:{\tiny $V_{10}$},text width=2mm] (n10) at (16,0)  {};
  \node [label=below:{\tiny $V_{11}$},text width=2mm] (n11) at (6,-5)  {};
  \node [label=below:{\tiny $V_{12}$},text width=2mm] (n12) at (4,-5)  {};
  \node [label=below:{\tiny $V_{13}$},text width=2mm] (n13) at (2,-5)  {};
  \node [label=below:{\tiny $V_{14}$},text width=2mm] (n14) at (0,-5)  {};
  \foreach \from/\to in {n1/n5,n1/n6,n1/n7,n1/n12,n1/n13,n1/n14,n2/n5,n2/n8,n2/n9,n2/n11,n2/n13,n2/n14,n3/n6,n3/n8,n3/n10,n3/n11,n3/n12,n3/n14,n4/n7,n4/n9,n4/n10,n4/n11,n4/n12,n4/n13, n5/n13,n5/n14,n6/n12,n6/n14,n7/n12,n7/n13,n8/n11,n8/n14,n9/n11,n9/n13,n10/n11,n10/n12}
    \draw (\from) -- (\to);
\end{tikzpicture}
\end{minipage}
\hspace{5cm}
\begin{minipage}{.25\textwidth}
\begin{tikzpicture}
  [scale=0.25,every node/.style={circle,fill=black,inner sep=0pt},very thick]
  \node [label=above:{\tiny $V_{1}$},text width=2mm] (n1) at (24,3) {};
  \node [label=above:{\tiny $V_{2}$},text width=2mm] (n2) at (26,3)  {};
  \node [label=above:{\tiny $V_{3}$},text width=2mm] (n3) at (28,3)  {};
  \node [label=above:{\tiny $V_{4}$},text width=2mm]  (n4) at (30,3)  {};
  \node [label=above:{\tiny $V_{5}$},text width=2mm] (n5) at (35,5)  {};
  \node [label=above:{\tiny $V_{6}$},text width=2mm] (n6) at (40,5)  {};
  \node [label=left:{\tiny $V_{7}$},text width=2mm] (n7) at (32,0)  {};
  \node [label=right:{\tiny $V_{8}$},text width=2mm] (n8) at (42,0)  {};
  \node [label=below:{\tiny $V_{9}$},text width=2mm] (n9) at (35,-5)  {};
  \node [label=below:{\tiny $V_{10}$},text width=2mm] (n10) at (40,-5)  {};
  \node [label=below:{\tiny $V_{11}$},text width=2mm] (n11) at (24,-3)  {};
  \node [label=below:{\tiny $V_{12}$},text width=2mm] (n12) at (26,-3)  {};
  \node [label=below:{\tiny $V_{13}$},text width=2mm] (n13) at (28,-3)  {};
  \node [label=below:{\tiny $V_{14}$},text width=2mm] (n14) at (30,-3)  {};

  \foreach \from/\to in {n1/n11,n2/n12,n3/n13,n4/n14,n5/n6,n5/n7,n5/n8,n5/n9,n5/n10,n6/n7,n6/n8,n6/n9,n6/n10,n7/n8,n7/n9,n7/n10,n8/n9,n8/n10,n9/n10/}
    \draw (\from) -- (\to);
\end{tikzpicture}
\end{minipage}

 \caption{ $In(R)$ and $In(R)_{SR}$} \label{figure:fr1}
\end{figure}
\end{center}
}
\end{example}
Then  $In(R)$  and  $In(R)_{SR}$ are shown in Figure \ref{figure:fr1}.

\begin{lem}\label{dimprod3}
Let  $n\geq 3$ be a positive integer and $R\cong \prod_{i=1}^{n}\mathbb{F}_{i}$, where $\mathbb{F}_{i}$ is a field for every $1\leq i \leq n$. Then
$\beta(In(R)_{SR})= 2n-3$.
\end{lem}
\begin{proof}
{
By Lemma \ref{dimprod2}, $In(R)_{SR}= H + \underbrace{K_{2}+ \cdots +K_{2}}_{n\, times}$, so $\beta(In(R)_{SR})= n + \beta (H)$. We show that $\beta (H)=n-3$.\\
By the proof of Lemma \ref{dimprod2}, $V(H)=\cup _{i=2}^{n-2}A_{i}$, where $A_{i}=\lbrace I \in V(In(R)_{SR})| NZC(I)=i\rbrace$.\\
Take the following facts into observation:

{\bf{Fact 1.}} Let $I,J\in A_i$, $I\neq J$ and $2 \leq i\leq n-2$. Then  since $NZC(I)=NZC(J)$, $IJ\notin E(In(R))$.

{\bf{Fact 2.}} Let $I,J\in A_i$, for some $2 \leq i\leq n-2$. If $I$ is not  adjacent to $J$  in $In(R)_{SR}$,  then by Fact (1) and  Lemma \ref{lemma2g}, $I\sim J^{c}$ in $In(R)$.

{\bf{Fact 3.}} Let $ i=\dfrac{n}{2}$, where $n$ is even. Then $In(R)_{SR}[A_i]$ is a complete graph, by Facts 1,2.

{\bf{Fact 4.}} Let $2 \leq i\leq [\dfrac{n}{2}]-1$, for even $n$ and $2 \leq i \leq [\dfrac{n}{2}]$, otherwise. Then $S_i\subseteq A_i$ is the largest subset of $A_i$ such that $IJ\notin E(In(R)_{SR}) $, for  every $I,J\in S_i$ (Indeed, $S_i$ is the largest independent subset of $A_i$ in $In(R){SR}[A_i]$). Then $|S_i|=[\dfrac{n}{i}]$. And symmetrically, for every $[\dfrac{n}{2}]+1 \leq i\leq n-2$, $S_i\subseteq A_i$ is the largest subset of $A_i$ such that $IJ\notin E(In(R)_{SR}) $, for  every $I,J\in S_i$ (Indeed, $S_i$ is the largest independent subset of $A_i$ in $In(R){SR}[A_i]$ ). Then $|S_i|=[\dfrac{n}{n-i}]$.\\
For every $I, J \in S_{i}$ (except for  $ i=\dfrac{n}{2}$, with even $n$ ), where $2 \leq i\leq n-2$, we have $I\sim J^{c}$ in $In(R)$. Thus for every $2 \leq j\neq i\leq n-2$ and for every $V\in S_{j}$, $V$ is adjacent to some vertices contained in $S_{i}$. In this case $|S_{2}|=|S_{n-2}|=[\dfrac{n}{2}]$ is the largest independent subset of $V(H)$. If some set $S^{\prime}$ is the largest independent subset of $V(H)$, then for every $I, J \in S^{\prime}$, either $I\sim J^{c}$  or $I\sim J$ in $In(R)$. First suppose that $S^{\prime}$ is the largest subset of $V(H)$ such that for every $I, J \in S^{\prime}$, $I\sim J^{c}$ in $In(R)$, then it is not hard to check that $|S^{\prime}|\leq |S_{2}|$.\\
Now let $W= \lbrace I_{1}= \mathbb{F}_{1}\times \mathbb{F}_{2}\times 0\times \cdots \times 0, I_{2}=\mathbb{F}_{1}\times \mathbb{F}_{2}\times \mathbb{F}_{3}\times 0 \times \cdots \times 0, \ldots, I_{n-3}= \mathbb{F}_{1}\times\cdots\times \mathbb{F}_{n-2}\times 0 \times 0 \rbrace$. Then $W$ is an independent subset of $V(H)$. For every $V \in V(H)\setminus W$, $V$ is adjacent to $I_{i}$ for some $1 \leq j\neq i\leq n-3$. Thus $W$ is the largest independent subset of $V(H)$. Since $|W|= n-3 \geq [\dfrac{n}{2}]$,  $\beta (H)=n-3$. Thus $\beta(In(R)_{SR})=n+\beta (H)= 2n-3$.  }
\end{proof}

Now, we are in a position to find $sdim(In(R))$.

\begin{thm}\label{dimprod4}
Suppose that $R\cong \prod_{i=1}^{n}\mathbb{F}_{i}$, where $\mathbb{F}_{i}$ is a field for every $1\leq i \leq n$ and $n\geq 3$ be a positive integer.  Then  $sdim(In(R))=2^{n}-2n +1$.
\end{thm}

\begin{proof}
{
By Lemma \ref{dimprod3}, $\beta(In(R)_{SR})=2n-3$. On the other hand, since
$|V(In(R)_{SR})|=2^{n}-2$, Gallai's theorem and Lemma \ref{Oellermann} show that $sdim(In(R))=|V(In(R)_{SR})|-\beta(In(R)_{SR}) =2^{n}-2n +1$.}
\end{proof}

\section{$dim_{M}(In(R))$ and $sdim(In(R))$, when $R$ is non-reduced}
 

As it has been mentioned in Section $2$, we consider rings $R$ with finitely many ideals. Then there exists positive integer $m$ such that $R\cong R_1\times\cdots \times R_m$, where $(R_{i},m_{i})$ is a local Artinian ring, for all $1\leq i \leq m$. If every $m_{i}$ is principal, then by \cite[Propostion 8.8]{ati}, every $R_{i}$ is a PIR with finitely many ideals. Moreover, ideals of every $R_{i}$ are totally ordered by inclusion. In this section, we compute $dim_{M}(In(R))$ and $sdim(In(R))$ for such rings $R$.

\begin{thm}\label{refe}
Suppose that  $R\cong \prod_{i=1}^{m}{R}_{i}$, where $R_{i}$ is a PIR non-field for every $1\leq i\leq m$ and  $m\geq 2$ is a positive integer. Then $dim_M(In(R))= (\sum_{i=1}^{m} {n}_{i}) + m-1$, where $n_{i}=|I(R_{i})|$ for every $1\leq i\leq m$.
\end{thm}
\begin{proof}
{We show that $dim_M(In(R))=(\sum_{i=1}^{m} {n}_{i}) + m-1$, where $n_{i}=|I(R_{i})|$ for every $1\leq i\leq m$. Indeed, we have the following claims:\\
\textbf{Claim 1.} $dim_M(In(R))\geq (\sum_{i=1}^{m} {n}_{i}) + m-1$.\\
 Let 
  $\chi_{1}=\lbrace I_{1i} \times 0 \times \cdots \times 0 \mid 0\neq I_{1i}\trianglelefteq R_{1}, 1\leq i\leq n_{1}+1\rbrace   $,\\
 
$\chi_{2}=\lbrace 0\times I_{2i} \times 0 \times \cdots \times 0 \mid 0\neq I_{2i}\trianglelefteq R_{2}, 1\leq i\leq n_{2}+1\rbrace $,\\
 \vdots
 
 and $\chi_{m}=\lbrace 0 \times 0 \times \cdots \times 0 \times I_{mi} \mid 0\neq I_{mi}\trianglelefteq R_{m}, 1\leq i\leq n_{m}+1\rbrace $.\\
From the above sets, only one member may not contained in $W$, otherwise, we assume that two vertices $J_{1}$ and $J_{2}$ of the above sets are not in $W$. If two vertices $J_{1}$ and $J_{2}$ are contained $\chi_{i}$, then  clearly $D(J_{1}|W)=D(J_{2}|W)$. Thus without loss of generality, assume that $J_{1}\in \chi_{1} $ and $J_{2}\in \chi_{2} $. In this case, there are $1\leq i\leq n_{1}$ and $1\leq j\leq n_{2}$ such that $J_{1}= I_{1i} \times 0 \times \cdots \times 0 $ and $J_{2}=  0 \times I_{2j} \times 0 \times \cdots \times 0 $. Suppose that $V= R_{1} \times I_{2j}\times 0 \times \cdots \times 0$, then $D(J_{1}|W)=D(V |W)$.\\
 Also, $|\bigcup_{i=1}^{m}\chi_{i}|= n_{1}+ n_{2}+ \dots +n_{m}+m$ implies that $dim_M(In(R))\geq (\sum_{i=1}^{m} {n}_{i}) + m-1$.\\
 
\textbf{Claim 2.}
$dim_M(In(R))\leq (\sum_{i=1}^{m} {n}_{i}) + m-1$.\\
Let $W=\bigcup_{i=1}^{m}\chi_{i}\setminus \lbrace I_{1i} \times 0 \times \cdots \times 0 \rbrace$. We claim that $W$ is a resolving set for $In(R)$. For this, let $I , J \not \in W $ and $I_{1}\times \cdots \times I_{m}=I\neq J=J_{1}\times \cdots \times J_{m}$. Hence, there exists  $1\leq j\leq m$ such that $I_{j}\neq J_{j}.$ Since $R$ is PIR,  $I_{j}\subset J_{j}$ or $J_{j}\subset I_{j}$. Without loss of generality assume that, $I_{j}\subset J_{j}$. In this case, we have $d(I,V)=2\neq1= d(V,J)$, where $V=0 \times \cdots \times 0 \times J_{j}\times 0\times \dots \times 0 $. Hence for every $I\neq J$, $D(I|W)\neq D(J|W)$. Thus $dim_M(In(R))\leq (\sum_{i=1}^{m} {n}_{i}) + m-1$.\\
By Claims 1,2, $dim_M(In(R))= (\sum_{i=1}^{m} {n}_{i}) + m-1.$}
\end{proof}
\begin{thm}\label{isomorphismm}
 Let  $R\cong S\times T$ such that $S= \prod_{i=1}^{m}{R}_{i}$, $m\geq 1$, $T=\prod_{j=1}^{n}\mathbb{F}_{j}$, $n\geq 1$, where $R_{i}$ is a PIR non-field for every $1\leq i\leq m$, $m\geq 1$ is a positive integer and $\mathbb{F}_{j}$ is a field  for every $1\leq j\leq n$, $n\geq 1$ is a positive integer. Then\\
 $1)$ If $m=n=1$, then $dim_M(In(R))=n+m+n_{1}-2=n_{1}$, where $n_{1}=|I(R_{1})|$.\\
 $2)$ If $m=1$ and $n=2$, then $dim_M(In(R))=n + m + n_{1}-1=n_{1}+2$, where $n_{1}=|I(R_{1})|$.\\
 $3)$ If $n\geq 3$, then $dim_M(In(R))=(\sum_{i=1}^{m} {n}_{i})+m+n$, where $n_{i}=|I(R_{i})|$.
\end{thm}
\begin{proof}
{
$1)$
We first show that $dim_M(In(R))\geq n_{1}$. For this, let $W$ be a metric basis for $In(R)$.\\
Let $I_{i}\subset I_{i+1}$, where  $I_{i}\in I(R_{1})$, for every $1\leq i\leq n_{1}-1$. Then, $d(I_{i} \times \mathbb{F}, V)=d(I_{i+1} \times \mathbb{F}, V)$, for every $ 1\leq i\leq n_{1}-1$, for every $V\in V(In(R))\setminus \lbrace I_{2} \times 0, \ldots , I_{n_{1}}\times 0\rbrace$. Thus $\lbrace W_{1}= I_{2} \times 0, W_{2}= I_{3} \times 0 ,\ldots , W_{n_{1}-1}=I_{n_{1}} \times 0\rbrace\subseteq W$. Also, for every $1\leq j\leq n_{1}-1$, we have $d(I_{n_{1}} \times \mathbb{F}, W_{j})=d(R_{1} \times 0, W_{j})$. Hence $W_{n_{1}}= R_{1}\times 0 \in W$. Therefore, $dim_M(In(R))\geq n_{1}$.\\
 Conversely, we show that $dim_M(In(R))\leq n_{1}$.\\ Let $W=\lbrace W_{1}= I_{2} \times 0, W_{2}= I_{3} \times 0 ,\ldots , W_{n_{1}-1}=I_{n_{1}} \times 0,W_{n_{1}}= R_{1}\times 0 \rbrace $.\\
 It is enough to show that $W$ is a resolving set of $In(R)$. Let $I^{\prime}=I_{1}^{\prime}\times I_{2}^{\prime}$ and $J=J_{1}\times J_{2}$ be two distinct vertices of $V(In(R))\setminus W$. If $I^{\prime}= I_{1}\times 0$ or $J= I_{1}\times 0$, then obviously $D(I^{\prime}|W)\neq D(J|W)$. Thus we may assume that $I_{2}^{\prime}=J_{2}=\mathbb{F}$.\\
 Since $I^{\prime}\neq J$, without loss of generality, we may assume that $I^{\prime}\subset J$. In this case there exists $1\leq i\leq n_{1}$ such that  $d(I^{\prime},W_{i})=2\neq 1= d(J, W_{i})$. Thus in both cases we have $D(I^{\prime}|W)\neq D(J|W)$. Therefore, $dim_M(In(R))\geq n_{1}$.\\
$2)$ Let $S= \lbrace V_{1}=I_{11} \times 0 \times 0, \ldots , V_{n_{1}}=I_{1n_{1}} \times 0 \times 0 , V_{n_{1}+1}=R_{1} \times 0 \times 0, V_{n_{1}+2}=0\times \mathbb{F}_{1}\times 0, V_{n_{1}+3}=0\times 0\times \mathbb{F}_{2} \rbrace $. From  the set $S$, only one member may not contained in the metric basis $W$. For if not, we assume that two vertices $J_{1}$ and $J_{2}$ of the above set are not contained in $W$. If $J_{1} , J_{2} \in \lbrace V_{1}, \ldots , V_{n_{1}} \rbrace$ or  $J_{1}\in \lbrace V_{1}, \ldots , V_{n_{1}} \rbrace$ and $J_{2}=V_{n_{1}+1}$ or $J_{1}=V_{n_{1}+2}$ and $J_{2}=V_{n_{1}+3}$, then clearly $D(J_{1}|W)=D(J_{2}|W)$. Now without loss of generality, assume that  $J_{1}= V_{1}$ and $J_{2}= V_{n_{1}+2} $, then $D(I_{11}\times \mathbb{F}_{1}\times \mathbb{F}_{2} |W)=D(I_{11}\times 0 \times \mathbb{F}_{2}|W)$. Finally, assume that $J_{1}= V_{n_{1}+1}$ and $J_{2}= V_{n_{1}+2}$, in this case, $D(I_{11}\times \mathbb{F}_{1}\times \mathbb{F}_{2} |W)=D(I_{11}\times 0 \times \mathbb{F}_{2}|W)$. Since $|S|=n_{1}+3$, $dim_M(In(R))\geq n_{1}+2$. \\
 Conversely, let $W=\lbrace V_{2}, \ldots , V_{n_{1}+3} \rbrace$. We claim that  $W$  is a resolving set and consequently a metric basis for  $In(R)$. It is enough to show that for every two distinct vertices $V^{\prime}=V_{1}^{\prime}\times V_{2}^{\prime} \times V_{3}^{\prime}$ and $U=U_{1}\times U_{2}\times U_{3}$  of $V(In(R))\setminus W$, $D(V^{\prime}|W)\neq D(U|W)$.\\
Since $V^{\prime} \neq U$, we have the following cases:\\
\textbf{Case 1.} 
$V_{j}^{\prime}=0$ and $U_{j}= \mathbb{F}_{j}$ or $V_{j}^{\prime}=\mathbb{F}_{j}$ , $U_{j}= 0$ for some $2\leq j\leq 3$ .\\
 Without loss of generality, we may assume that $V_{j}^{\prime}=0$ and $U_{j}= \mathbb{F}_{j}$. This clearly implies that $d(V^{\prime}, X_{j})\in \lbrace 2, 3\rbrace \neq 1=d(U, X_{j}).$ Thus $D(V^{\prime}|W)\neq D(U|W)$.\\
 \textbf{Case 2.} 
$V_{1}^{\prime}\neq U_{1} $.\\
Since $R$ is a PIR, $V_{1}^{\prime}\subset U_{1}$ or $U_{1}\subset V_{1}^{\prime}$. Without loss of generality, we may assume that $V_{1}^{\prime}\subset U_{1}$. If $U_{1}=I_{11}$, then $d(U , I_{12}\times 0 \times 0)=2 \neq 3=d(V^{\prime},  I_{12}\times 0 \times 0)$, otherwise, $d(U , U_{1}\times 0 \times 0)=1 \neq 2=d(V^{\prime},  U_{1}\times 0 \times 0)$.\\
Thus in both cases we have that $D(V^{\prime}|W)\neq D(U|W)$. Therefore, $dim_M(In(R))\leq n_{1}+2$. 
$3)$ We first show that $dim_M(In(R))\geq (\sum_{i=1}^{m} {n}_{i})+m+n$.\\
For this let\\$\chi_{1}=\lbrace I_{1i} \times 0 \times \dots \times 0 \mid 0\neq I_{i1}\trianglelefteq R_{1}, 1\leq i\leq n_{1}+1\rbrace$,\\$\chi_{2}=\lbrace 0\times I_{2i} \times 0 \times \dots \times 0 \mid 0\neq I_{i1}\trianglelefteq R_{2}, 1\leq i\leq n_{2}+1\rbrace$,\\
 \vdots
 
and $\chi_{m}=\lbrace 0 \times 0 \times \dots \times 0 \times I_{mi} \mid 0\neq I_{i1}\trianglelefteq R_{m}, 1\leq i\leq n_{m}+1\rbrace$.\\
 Let $S= \cup_{X\in \chi_{i}, 1\leq i\leq m }^{}X \cup_{j=m+1}^{m+n} X_{j}$, $X_{j}=0 \times \dots \times 0 \times \mathbb{F}_{j-m}\times 0 \times \dots \times 0$ where $m+1\leq j\leq m+n$. We claim that for every metric basis $W$, we have $S\subseteq W$. Otherwise, there exists $J \in S$ such that $J\notin W$. Since $J \in S$, we have the following cases:\\ 
\textbf{Case 1.} $J= 0\times \cdots \times 0 \times I_{ij}\times 0\times \cdots \times 0 $, where $1\leq i \leq m$ and $1\leq j \leq n_{i}$.\\
In this case, if $j=1$, then $D(R_{1}\times R_{2} \times \cdots \times R_{i-1}\times J_{i1}\times R_{i+1}\times \cdots \times \mathbb{F}_{m+n-1} \times 0|W)=D(R_{1}\times R_{2} \times \cdots \times R_{i-1}\times 0\times R_{i+1}\times \cdots \times \mathbb{F}_{m+n-1} \times 0|W)$. If $j=n_{i}$, then $D(0\times \cdots  \times 0 \times I_{in_{i}}\times 0 \times \cdots \times 0 \times \mathbb{F}_{m+1}\times 0 \times \cdots \times 0|W)=D(0\times \cdots  \times 0 \times I_{in_{i}-1}\times 0 \times \cdots \times 0 \times \mathbb{F}_{m+1}\times 0 \times \cdots \times 0|W)$. Otherwise, $D(0\times \cdots  \times 0 \times I_{in_{i}}\times 0 \times \cdots \times 0 \times \mathbb{F}_{m+1}\times 0 \times \cdots \times 0|W)=D(0\times \cdots  \times 0 \times R_{i}\times 0 \times \cdots \times 0 \times \mathbb{F}_{m+1}\times 0 \times \cdots \times 0|W)$.\\
\textbf{Case 2.} $J= 0\times \cdots \times 0 \times R_{i}\times 0 \times \cdots \times 0 $, where $1\leq i \leq m$.\\
In this case, $D(0\times \cdots  \times 0 \times I_{in_{i}}\times 0 \times \cdots \times 0 \times \mathbb{F}_{m+1}\times 0 \times \cdots \times 0|W)=D(0\times \cdots  \times 0 \times R_{i}\times 0 \times \cdots \times 0 \times \mathbb{F}_{m+1}\times 0 \times \cdots \times 0|W)$.\\
\textbf{Case 3.} $J=X_{j}$ (See Remark \ref{dimfin}), where $m+1\leq j \leq m+n$\\
In this case, $D(I_{11}\times R_{2}\times \cdots \times R_{m}\times \mathbb{F}_{1} \times \cdots \times \mathbb{F}_{j-1} \times 0 \times \mathbb{F}_{j+1}\times \cdots \times \mathbb{F}_{m+n-1}\times 0|W)=D(I_{11}\times R_{2}\times \cdots \times \mathbb{F}_{m+n-1}\times 0|W)$.\\
Thus $S\subseteq W$. This implies that $dim_M(In(R))=|W|\geq |S|=(\sum_{i=1}^{m} {n}_{i})+m+n$.\\
Conversely, we claim that  $W= \cup_{X\in \chi_{i}, 1\leq i\leq m }^{}X \cup_{j=m+1}^{m+n} X_{j}$  is a resolving set and consequently a metric basis for $In(R)$. It is enough to show that for every two distinct vertices $V^{\prime}=V_{1}^{\prime}\times \cdots \times V_{m+n}^{\prime}$ and $U=U_{1}\times \cdots \times U_{m+n}$  of $V(In(R))\setminus W$,  the inequality $D(V^{\prime}|W)\neq D(U|W)$ holds.\\
Since $V^{\prime} \neq U$, we have the following cases:\\
\textbf{Case 1.} 
$V_{j}^{\prime}=0$ and $U_{j}= \mathbb{F}_{m-j}$ or $V_{j}^{\prime}=\mathbb{F}_{m-j}$ , $U_{j}= 0$ for some $m+1\leq j\leq m+n$ .\\
 Without loss of generality, one may assume that $V_{j}^{\prime}=0$ and $U_{j}= \mathbb{F}_{m-j}$. This clearly implies that $d(V^{\prime}, X_{j})\in \lbrace 2, 3\rbrace \neq 1=d(U, X_{j}).$ Thus $D(V^{\prime}|W)\neq D(U|W)$.\\
 \textbf{Case 2.} 
$V_{i}^{\prime}\neq U_{i}$ for some $1\leq i\leq m $.\\
Since $R$ is a PIR, $V_{i}^{\prime}\subset U_{i}$ or $U_{i}\subset V_{i}^{\prime}$. Without loss of generality, one may assume that $V_{i}^{\prime}\subset U_{i}$. Then\\ $d(U , 0\times \dots \times 0 \times U_{i}\times 0\times \dots \times 0)=1 \neq 2=d(V^{\prime},  0\times \dots \times 0 \times U_{i}\times 0\times \dots \times 0).$\\
Thus in both cases we have that $D(V^{\prime}|W)\neq D(U|W)$. Therefore, $dim_M(In(R))\leq (\sum_{i=1}^{m} {n}_{i}) +m+n$.
}
\end{proof}

Next, we study $sdim(In(R))$. First, the case no fields appear in decomposition of $R$ is investigated.
\begin{lem}\label{dimprod0}
Let $m\geq 2$ be a positive integer and $R \cong \prod_{i=1}^{m}R_{i}$, where $R_{i}$ is a PIR for every $1\leq i \leq m$. For every  $I, J \in V(In(R))$, $d(I,J)=3$ if and only if $I, J \in M$ and $I=J^{c}$.
\end{lem}
\begin{proof}
{
First let $I, J \in M$ and $I=J^{c}$. Since $I \nsim J$, $d(I,J)\geq 2$. Now let $d(I,J)=2$. Thus there exists a vertex $V$ such that $I\sim V \sim J$ is the shortest path between $I$ and $J$. Since  $I\sim V$, $I\subset V$ or $V\subset I$. If $V\subset I$, then clearly $V \nsim J$. Thus $I\subset V$. A similar argument shows that $J\subset V$, which is impossible. Thus $d(I,J)=3$.\\
Conversely, since $d(I,J)=3$, if $I_{i}\neq 0$, then $J_{i}= 0$ and if  $J_{i}\neq 0$, then $I_{i}=0$ for every $1\leq i \leq m$. Otherwise, $I\sim V \sim J$, where $V=0 \times \cdots \times 0 \times V_{i}\times 0\times \cdots \times 0$, $V_{i}= I_{i}$ if $I_{i}\subset J_{i}$, else $V_{i}=J_{i}$. This implies that $d(I,J)\leq 2$, a contradiction.   
 Now suppose to the contrary, $I, J \notin M$. Then there exists $1\leq i \leq m$ such that $I_{i}\in I(R_{i})$. Thus $I \sim V \sim J$, where $V= R_{1}\times \cdots \times R_{i-1}\times I_{i}\times R_{i+1}\times \cdots\times R_{m}$, which implies that $d(I,J)\leq 2$, a contradiction. Thus $I, J \in M$ and $I=J^{c}$. 
}
\end{proof}
\begin{lem}\label{dimprod00}
Let $n\geq 2$ be a positive integer and $R \cong \prod_{i=1}^{m}R_{i}$, where $R_{i}$ is a PIR non-field for every $1\leq i \leq m$. Then, $IJ^{c}\in In(R)$ if and only if $JI^{c} \in In(R)$, for every  $I, J \in M $.
\end{lem}
\begin{proof}
{
It is straight forward.
}

\end{proof}
\begin{lem}\label{dimprod4}
Suppose that  $R\cong \prod_{i=1}^{m}{R}_{i}$, where $R_{i}$ is a PIR  non-field for every $1\leq i\leq m$, $m\geq 2$ is a positive integer. Then the following statements hold:\\
$1)$ $V(In(R))=V(In(R)_{SR})$.\\
$2)$ Suppose that $I, J \in M \subset V(In(R)_{SR})$, then  $IJ \in E(In(R)_{SR})$ if and only if $I= J^{c}$ or $IJ, IJ^{c}\notin In(R)$.\\
$3)$ Suppose that $I, J \in V(In(R)_{SR})\setminus M$, then  $IJ \in E(In(R)_{SR})$ if and only if $IJ \notin E(In(R)).$\\
$4)$ For every $I\in V(In(R)_{SR})\setminus M$ and $J\in M$, $IJ \in E(In(R)_{SR})$ if and only if $IJ, IJ^{c} \notin E(In(R))$.
\end{lem}
\begin{proof}{
$1)$ For every $I=I_{1}\times\cdots \times I_{n}\in M$, by the proof of Lemma \ref{lemma2d}, $d(I,I^{c})=3=diam(In(R)).$ Thus $I, I^{c}$ are mutually maximally distant and so $I\in V(In(R)_{SR})$. Also for every $I \in V(In(R))\setminus M$, there exists $J \in V(In(R))\setminus M$ such that $I, J$ are mutually maximally distant and $I\in V(In(R)_{SR})$ i.e., $V(In(R))=V(In(R)_{SR}).$\\
$2)$ If $I= J^{c}$ or $IJ, IJ^{c}\notin In(R)$, then clearly $I, J$ are mutually maximally distant and $IJ\in E(In(R)_{SR})$.\\
Now suppose that $IJ\in E(In(R)_{SR})$ and $I\neq J^{c}$. Since $I\neq J^{c}$, $d(I,J)\leq 2$. Also, $I, J$ are mutually maximally distant, thus $IJ, IJ^{c}\notin In(R)$.\\
$3)$ If $IJ \in E(In(R)_{SR})$, then  $I, J$ are mutually maximally distant, thus clearly $IJ\notin E(In(R))$.\\
Now suppose that $IJ \notin E(In(R))$. Since $d(I,J)_{In(R)}\leq 2$, for every $I\in V(In(R)_{SR})\setminus M$ and for every $J\in V(In(R)_{SR})$, we deduce that $I, J$ are mutually maximally distant and  $IJ\in E(In(R)_{SR})$. \\
$4)$ One side is clear. To prove the other side, assume that $IJ, IJ^{c} \notin E(In(R))$. Since $d(I,J)\leq 2$, $I, J$ are mutually maximally distant and  $IJ\in E(In(R)_{SR})$.    
}
\end{proof}
\begin{remark}
\label{dimfin1}
Let $R \cong R_{1}\times R_{2}$. If  $|I(R_{1})|=|I(R_{2})|=1$, then $In(R)_{SR}= K_{3}+ K_{2}+ K_{2}$. Thus the condition $|I(R_{i})|\geq 2$ is  necessary for $In(R)_{SR}$ to be connected. Therefore, we exclude this case from Lemma $\ref{dimprod5}$. 
\end{remark}

\begin{lem}\label{dimprod5}
Suppose that  $R\cong \prod_{i=1}^{m}{R}_{i}$, where $R_{i}$ is a PIR non-field for every $1\leq i\leq m$ and $m\geq 2$ is a positive integer. Then $In(R)_{SR}$ is a connected graph.
\end{lem}
\begin{proof}{
  By Lemma \ref{dimprod4}, $V(In(R))=V(In(R)_{SR})$. We show that for any two distinct vertices $X'=(X'_1,\ldots,X'_m)$ and $Y'=(Y'_1,\ldots,Y'_m)$ there is a path between them. For this we have the following cases:\\
 \textbf{Case 1.} $X', Y'\in M$.\\
 If $X',=Y'^{c}$ or $X'Y', X'Y'^{c} \notin E(In(R))$, then by Lemma \ref{dimprod4}, $X'$ and $Y'$ are adjacent in $In(R)_{SR}$. Thus suppose that $X'Y'\notin E(In(R)_{SR})$, so either $X'\neq Y'^{c}$, $Y'X'^{c}\in E(In(R))$ or $X'\neq Y'^{c}$, $X'Y'\in E(In(R))$. If $X'\neq Y'^{c}$ and $X'Y'\in E(In(R))$, then $X' \subset Y'$ or $Y' \subset X'$. Without loss of generality, we may assume that $X'\subset Y'$. Thus there are $1 \leq i\neq j \leq m$ such that $X'_{i}=Y'_{i}=R_{i}$ and  $X'_{j}=Y'_{j}=0$. Let $V= V_{1} \times \cdots \times V_{m}$, where $V_{i}= I_{i1}$, $V_{j}= I_{j1}$ and the other components are zero. Then $X' \sim V \sim Y'$ is a path between $X'$ and $Y'$. Now let $X'\neq Y'^{c}$ and $Y'X'^{c}\in E(In(R))$. Since $Y'X'^{c}\in E(In(R))$, there are $1 \leq i\neq j \leq m$ such that $X'_{i}=0, Y'_{i}=R_{i}$ and  $X'_{j}=R_{j}, Y'_{j}=0$. Let $V= V_{1} \times \cdots \times V_{m}$, where $V_{i}= I_{i1}$, $V_{j}= I_{j1}$ and the other components are $R_{i}^{,}$s. Then $X' \sim V \sim Y'$ is a path between $X'$ and $Y'$. Thus there exists a path between $X'$ and $Y'$.\\
 \textbf{Case 2.} $X', Y' \in V(In(R)_{SR})\setminus M$.\\
 If $X'Y'\notin E(In(R))$, by Lemma \ref{dimprod4}, $X'Y'\in E(In(R)_{SR})$. Thus suppose that $X'Y'\in E(In(R))$, so $X' \subset Y'$ or $Y' \subset X'$. Without loss of generality, we may assume that $X'\subset Y'$. If there exists $1 \leq i \leq m$ such that $X'_{i}=Y'_{i}=0$, then $X' \sim 0\times \cdots \times0 \times I_{i1}\times 0 \times \cdots \times 0 \sim Y'$ is a path between $X'$ and $Y'$. Thus suppose that $Y'_{i}\neq 0$ for every $1 \leq i \leq m$. Since $Y'\notin M$ there exists $1 \leq i \leq m$ such that $Y'_{i}\neq R_{i}$, so $Y' \sim V$, where $V$ is the vertex that is obtained by replacing the first and the $i$-th component of $Y'$ with zero and $R_{i}$. If the vertex $V$ is adjacent to $X'$, then $Y' \sim V\sim X'$ is a path between $X'$ and $Y'$. Otherwise, continue the same process to get a path between $X'$ and $Y'$.\\
 \textbf{Case 3.} $X'\in V(In(R)_{SR})\setminus M$ and $Y'\in M$.\\
 If $X'Y', X'Y'^{c}\notin E(In(R))$, then by Lemma \ref{dimprod4}, $X'Y' \in E(In(R)_{SR})$. Thus suppose that $X'Y'\in E(In(R))$ or $X'Y'^{c}\in E(In(R))$. If $X'Y'\in E(In(R))$, then $X' \subset Y'$ or $Y' \subset X'$. Without loss of generality, we may assume that $X'\subset Y'$. Since $Y'\in M$ there exists $1 \leq i\neq j \leq m$ such that $Y'_{i}=0$ and $Y'_{j}\neq0$. Then $X'\sim V\sim Y'$, where $V=0\times \cdots \times 0\times R_{i}\times 0 \times \cdots \times 0 \times I_{j1}\times 0 \times \cdots \times 0$ is a  path between $X'$ and $Y'$. If $X'Y'^{c}\in E(In(R))$, by a similar argument there exists a  path between $X'$ and $Y'$. Thus $In(R)_{SR}$ is a connected graph. 
}
\end{proof}
\begin{lem}\label{dimprod6}
 
Suppose that  $R\cong \prod_{i=1}^{m}{R}_{i}$, where $R_{i}$ is a PIR  non-field for every $1\leq i\leq m$ and $m\geq 2$ is a positive integer. Then $\beta(In(R)_{SR})=\Sigma _{i=1}^{m}n_{i} + m -1$.
\end{lem}
\begin{proof}
{
If $m=2$ and $|I(R_{1})|=|I(R_{2})|=1$, then $In(R)_{SR}= K_{3}+ K_{2}+ K_{2}$, so $\beta(In(R)_{SR})= 3 = n_{1}+ n_{2}+m-1$. Otherwise, by Lemma \ref{dimprod5}, $In(R)_{SR}$ is a connected graph. Moreover, by Lemma \ref{dimprod4}, $V(In(R))=V(In(R)_{SR})$. Suppose that $X, Y \in V(In(R)_{SR})\setminus M$, then by Lemma \ref{dimprod4}, $XY \in E(In(R)_{SR})$ if and only if $XY \notin E(In(R))$. So the vertex $X$ is not adjacent to $Y$ if and only if $X\subset Y$ or $Y \subset X$. Let\\
$S= \lbrace 0\times \cdots \times 0\times I_{m,1}, \ldots , 0\times \cdots \times 0\times I_{m,n_{m}},\dots,  0\times \cdots \times 0\times I_{m-1,1}\times R_{m}, \ldots, 0\times \cdots \times 0\times I_{m-1,n_{m-1}}\times R_{m}, \ldots, I_{1,1}\times R_{2}\times \cdots \times R_{m}, \ldots , I_{1,n_{1}}\times R_{2}\times \cdots \times R_{m} \rbrace$. Then $S$ is the largest independent subset of $V(In(R)_{SR})\setminus M$ such that $X\nsim Y$ for every $X, Y \in S$ and $|S|=\Sigma_{i=1}^{m}n_{i}$.  Indeed, $S$ is the largest independent subset of $V(In(R)_{SR})\setminus M$  in $In(R)_{SR}[V(In(R)_{SR})\setminus M]$. On the other hand, among the vertices contained in $M$, only vertices which are contained in  $S^{\prime}=\lbrace 0\times \cdots \times 0 \times R_{m}, 0\times \cdots \times 0\times R_{m-1} \times R_{m}, \dots, 0\times R_{2}\times \cdots \times R_{m} \rbrace $ are not adjacent to vertices in $S$( we note that $|S^{\prime}|=m-1$). Let $A= S \cup S^{\prime}$. Thus $A$ is the largest independent subset of $V(In(R)_{SR})$ and hence $\beta (In(R)_{SR})=|A|= \Sigma _{i=1}^{m}n_{i} + m -1 $.
}
\end{proof}

\begin{thm}\label{isomorphism3}
 
Suppose that  $R\cong \prod_{i=1}^{m}{R}_{i}$, where $R_{i}$ is a PIR  non-field for every $1\leq i\leq m$ and $m\geq 2$ is a positive integer. Then $sdim(In(R))=\Pi_{i=1}^{m}(n_{i}+2)-\Sigma _{i=1}^{m}n_{i} - m -1$.
\end{thm}
\begin{proof}
{ By Lemma \ref{dimprod6}, $\beta(In(R)_{SR})=\Sigma _{i=1}^{m}n_{i} + m -1$. On the other hand, since
$|V(In(R)_{SR})|=\Pi_{i=1}^{m}(n_{i}+2)-2$, Gallai$^{^,}$s theorem and Lemma \ref{Oellermann} show that $sdim(In(R))=|V(In(R)_{SR})|-\beta(In(R)_{SR}) =\Pi_{i=1}^{m}(n_{i}+2)-\Sigma_{i=1}^{m}n_{i} - m -1$. }
\end{proof}

Finally, we investigate $sdim(In(R))$, where both of fields and non-fields appear in decomposition of $R$.

\begin{lem}\label{dimprod7}
Let  $R\cong S\times T$ such that $S= \prod_{i=1}^{m}{R}_{i}$, $m\geq 1$, $T=\prod_{j=1}^{n}\mathbb{F}_{j}$, $n\geq 1$, where $R_{i}$ is a PIR  non-field for every $1\leq i\leq m$, $m\geq 1$ is a positive integer and $\mathbb{F}_{j}$ is a field  for every $1\leq j\leq n$, $n\geq 1$ is a positive integer. Then the following statements hold:\\
$1)$ If $m=1$, $V(In(R)_{SR})= V(In(R))\setminus \lbrace I_{1}\times 0\times \cdots \times 0, I_{1}\times \mathbb{F}_{1}\times \cdots \times \mathbb{F}_{n}\rbrace $, else $V(In(R)_{SR})=V(In(R))$.\\
$2)$ Suppose that $I, J \in M \subset V(In(R)_{SR})$, then  $IJ \in E(In(R)_{SR})$ if and only if $I= J^{c}$ or $IJ, IJ^{c}\notin In(R)$. In particular, for every $V\in V(In(R)_{SR})$, if $VX_{i}\in E(In(R)_{SR})$, then $V= X_{i}^{c}$.\\
$3)$ Suppose that $I, J \in V(In(R)_{SR})\setminus M$, then  $IJ \in E(In(R)_{SR})$ if and only if $IJ \notin E(In(R)).$\\
$4)$ For every $I\in V(In(R)_{SR})\setminus M$ and $J\in M$, $IJ \in E(In(R)_{SR})$ if and only if $IJ, IJ^{c} \notin E(In(R))$.
\end{lem}
\begin{proof}{
 Let $m=1$ and $X = \lbrace I_{1}\times 0 \times \cdots \times 0, I_{1}\times \mathbb{F}_{1}\times \cdots \times \mathbb{F}_{n} \rbrace$. It is not hard to check that for any $U \in X$, there is no $V \in V(In(R))$ such that $U$ and $V$ are mutually
maximally distant. Thus $V(In(R)_{SR})= V(In(R))\setminus X$. Now we show that for  every $V\in V(In(R)_{SR})$, if $VX_{i}\in E(In(R)_{SR})$, then $V= X_{i}^{c}$. Suppose to the contrary, $V\neq X_{i}^{c}$. Then $ X_{i}\subset V$ or $V X_{i}^{c}\in E(In(R))$. These two cases imply that $ VX_{i}\notin E(In(R)_{SR}$, a contradiction.
 To complete the proof, it is enough to apply a similar argument to that of Lemma \ref{dimprod4}.
}
\end{proof}
\begin{lem}\label{dimprod8}
Let  $R\cong S\times T$ such that $S= \prod_{i=1}^{m}{R}_{i}$, $m\geq 1$, $T=\prod_{j=1}^{n}\mathbb{F}_{j}$, $n\geq 1$, where $R_{i}$ is a PIR  non-field for every $1\leq i\leq m$, $m\geq 1$ is a positive integer and $\mathbb{F}_{j}$ is a field  for every $1\leq j\leq n$, $n\geq 1$ is a positive integer. Then $In(R)_{SR}=  H + \underbrace{K_{2}+ \cdots +K_{2}}_{n\, times} $, where $H$ is a connected graph.
\end{lem}
\begin{proof}{
 Since for every $I \in \lbrace X_{m+1}, \ldots , X_{m+n} \rbrace$, $I$ is only mutually maximally distant from $I^{c}$ and vice versa, $In(R)_{SR}$ contains $n$ copies of $K_{2}$. To complete the proof, it is enough to apply a similar argument to that of Lemma \ref{dimprod5} and Lemma \ref{dimprod2}.
}
\end{proof}
\begin{lem}\label{dimprod9}
 
Let  $R\cong S\times T$ such that $S= \prod_{i=1}^{m}{R}_{i}$, $m\geq 1$, $T=\prod_{j=1}^{n}\mathbb{F}_{j}$, $n\geq 1$, where $R_{i}$ is a PIR non-field  for every $1\leq i\leq m$, $m\geq 1$ is a positive integer and $\mathbb{F}_{j}$ is a field  for every $1\leq j\leq n$, $n\geq 1$ is a positive integer. Then the followings hold:\\
$1)$ If $m=1$, then $\beta(In(R)_{SR})=2n + n(R_{1})-2$.\\
$2)$ If $m\geq 2$, then $\beta(In(R)_{SR})=(\sum_{i=1}^{m} {n}_{i})+ 2n + m-1$.
\end{lem}
\begin{proof}
{
$1)$ By Lemma \ref{dimprod8}, $In(R)_{SR}=  H + \underbrace{K_{2}+ \cdots +K_{2}}_{n\, times}$, so $\beta(In(R)_{SR})= \beta(H)+ n$. Also, by a similar argument to that of Lemma \ref{dimprod6}  and  case $(1)$ of Lemma \ref{dimprod7}, $S= \lbrace I_{1,1}\times 0\times \cdots \times 0 , \ldots , I_{1,n(R_{1})}\times 0\times \cdots \times 0, I_{1,n(R_{1})}\times \mathbb{F}_{1} \times 0 \cdots \times 0, \ldots , I_{1,n(R_{1})}\times \mathbb{F}_{1} \times  \cdots \times \mathbb{F}_{n-1} \rbrace$ is  the largest
independent subset of $V(H)$ and $|S|= n(R_{1}) + n-2 $. Hence $\beta(In(R)_{SR})=|S| + n= n(R_{1}) + 2n -2$.\\
$2)$ By Lemma \ref{dimprod7}, $V(In(R)_{SR})=V(In(R))$ and by Lemma \ref{dimprod8}, $In(R)_{SR}=  H + \underbrace{K_{2}+ \cdots +K_{2}}_{n\, times}$, so $\beta(In(R)_{SR})= \beta(H)+ n$. Also, by a similar argument to that of Lemma \ref{dimprod6}  and Lemma \ref{dimprod3}, $S= \lbrace 0\times \cdots \times 0\times I_{m,1}\times 0\times \cdots \times 0, \ldots , 0\times \cdots \times 0\times I_{m,n_{m}}\times 0\times \cdots \times 0, 0\times \cdots \times 0\times R_{m}\times 0 \times \cdots \times 0, 0\times \cdots \times 0\times I_{m-1,1}\times R_{m}\times 0 \times \cdots \times 0, \ldots, 0\times \cdots \times 0\times I_{m-1,n_{m-1}}\times R_{m}\times 0 \times \cdots \times 0,0\times \cdots \times 0\times R_{m-1}\times R_{m} \times 0 \times \cdots \times 0 ,\ldots, I_{1,1}\times R_{2} \times \cdots \times R_{m}\times 0 \times \cdots \times 0, \ldots , I_{1,n_{1}}\times R_{2} \times 0 \times \cdots \times R_{m}\times 0 \times \cdots \times 0, I_{1,n_{1}}\times R_{2}\times \cdots \times R_{m}\times \mathbb{F}_{1} \times 0 \times \cdots \times 0, \ldots, I_{1,n_{1}}\times R_{2}\times \cdots \times R_{m}\times \mathbb{F}_{1} \times  \cdots \times \mathbb{F}_{n}   \rbrace$ is  the largest
independent subset of $V(H)$ and $|S|= (\sum_{i=1}^{m} {n}_{i})+ n + m-1$. Hence $\beta(In(R)_{SR})=|S|+ n=(\sum_{i=1}^{m} {n}_{i})+ 2n + m-1$
}
\end{proof}

We close this paper with the following result.

\begin{thm}\label{isomorphism4}
 Let  $R\cong S\times T$ such that $S= \prod_{i=1}^{m}{R}_{i}$, $m\geq 1$, $T=\prod_{j=1}^{n}\mathbb{F}_{j}$, $n\geq 1$, where $R_{i}$ is a PIR  non-field for every $1\leq i\leq m$, $m\geq 1$ is a positive integer and $\mathbb{F}_{j}$ is a field  for every $1\leq j\leq n$, $n\geq 1$ is a positive integer. Then\\
$1)$ If $m=1$, then $sdim(In(R))=(n_{1}+2)2^{n}-2n - n_{1}-2$.\\
$2)$ If $m\geq 2$, then $sdim(In(R))=\prod_{i=1}^{m}(n_{i}+2)2^{n}-(\sum_{i=1}^{m} {n}_{i})- 2n - m-1$.
\end{thm}
\begin{proof}
{
$1)$ By Lemma \ref{dimprod9}, $\beta(In(R)_{SR})=2n + n(R_{1})-2$. On the other hand, since
$|V(In(R)_{SR})|=n_{1}+2)2^{n}-4$, Gallai$^{^,}$s theorem and Lemma \ref{Oellermann} show that $sdim(In(R))=|V(In(R)_{SR})|-\beta(In(R)_{SR}) =(n_{1}+2)2^{n}-2n - n_{1}-2$.\\
$2)$ By Lemma \ref{dimprod9}, $\beta(In(R)_{SR})=(\sum_{i=1}^{m} {n}_{i})+ 2n + m-1$. Since
$|V(In(R)_{SR})|=\prod_{i=1}^{m}(n_{i}+2)2^{n}-2$, Gallai$^{^,}$s theorem and Lemma \ref{Oellermann} show that $sdim(In(R))=|V(In(R)_{SR})|-\beta(In(R)_{SR}) =\prod_{i=1}^{m}(n_{i}+2)2^{n}-(\sum_{i=1}^{m} {n}_{i})- 2n - m-1$.}
\end{proof}

{\bf Conflicts of interest/Competing interests:} There is no conflict of interest regarding the publication of this paper.

{\bf Ethics approval:} Not applicable.

{\bf Funding:} The author(s) received no financial support for the research, authorship, and/or publication of this article.

{\bf Availability of data and material (data transparency):} Not applicable.

{\bf Code availability (software application or custom code):} Not applicable.

{\bf Authors contribution:} All authors have contributed equally.

{\bf Acknowledgment:} Not applicable.

{\bf Human Participants and/or Animals:} Not applicable.


{}

\end{document}